\documentclass[12pt]{article}
\usepackage{amsmath,amssymb,amsthm}
\usepackage{amsfonts, amsmath, amsthm, amssymb}
\usepackage{epsfig}
\usepackage{color}

\title{Signed magic rectangles with three filled cells in each column}
\author {
Abdollah Khodkar, David Leach and Brandi Ellis\\
Department of Mathematics\\
University of West Georgia\\
Carrollton, GA 30118\\
{\tt akhodkar@westga.edu},
{\tt cleach@westga.edu},\\
{\tt bellis5@my.westga.edu}
}

\date{}

\newtheorem{prelem}{{\bf Theorem}}

 \newtheorem{theorem}{Theorem}
\newtheorem{corollary}[theorem]{Corollary}
\newtheorem{lemma}[theorem]{Lemma}
\newtheorem{remark}[theorem]{Remark}
\newtheorem{proposition}[theorem]{Proposition}
\newtheorem{maintheorem}[theorem]{Main Theorem}

\theoremstyle{definition}

\newtheorem{example}[theorem]{Example}

\begin{document}

\maketitle

\begin{abstract}
\noindent A {\em signed magic rectangle} $SMR(m,n;k, s)$ is an $m \times n$ array with entries from $X$, where
$X=\{0,\pm1,\pm2,\ldots, $ $\pm (mk-1)/2\}$ if $mk$ is odd and $X = \{\pm1,\pm2,\ldots,\pm mk/2\}$ if $mk$ is even,
such that precisely $k$ cells in every row and $s$ cells in every column are filled,
every integer from set $X$ appears exactly once in the array and
the sum of each row and of each column is zero. In this paper, we prove that a signed magic rectangle
$SMR(m,n;k, 3)$ exists if and only if  $3\leq m,k\leq n$ and $mk=3n$.
\end{abstract}

\section{Introduction}\label{SEC1}

A {\em magic rectangle} of order $m\times n$ (see \cite{KL2}) with precisely $r$ filled cells in each
row and precisely $s$ filled cells in each column, $MR(m, n; r, s)$,
is an arrangement of the numbers from 0 to $mr-1$ in an $m\times n$
array such that each number occurs exactly once in the array and the
sum of the entries of each row is the same and the sum of entries of
each column is also the same. If $r=n$ or $s=m$, then the array has no empty cells and we denote it by $MR(m,n)$.

The following theorem (see \cite {TH1, TH2, sun}) settles the existence of an $MR(m,n)$.

\begin{theorem}\label{TH:sun}
An $m \times n$ magic rectangle exists if and only if $m \equiv n \pmod 2$, $m + n > 5$, and $m, n > 1$.
\end{theorem}

An {\em integer Heffter array} $H(m, n; s, t)$ is an $m\times n$ array with entries from
$X=\{\pm1,\pm2,$ $\ldots,\pm ms\}$
such that each row contains $s$ filled cells and each column contains $t$ filled cells,
the elements in every row and column sum to 0 in ${\mathbb Z}$, and
for every $x\in A$, either $x$ or $-x$ appears in the array.
The notion of an integer Heffter array $H(m, n; s, t)$ was first defined by Archdeacon in \cite{arc1}.
Heffter arrays can be used for construction of orthogonal cycle systems or embeddings of pairs of cycle systems on surfaces.

Integer Heffter arrays $H(m, n; s, t)$
with $m=n$ represent a type of magic square where each number from the set
$\{1,2,\ldots,$ $ms\}$ is used once up to sign. A Heffter array is {\em tight} if it has no empty cell; that is,
$n=s$ (and necessarily $m = t$). We denote a signed magic rectangle $SMR(m; n; n;m)$
by $SMR(m; n)$.

\begin{theorem}\label{tightHeffter} {\rm\cite{arc2}}
Let $m, n$ be integers at least 3.
There is a tight integer Heffter array $H(m, n)$ if and only if $mn\equiv 0, 3 \pmod 4$.
\end{theorem}

A {\em square} integer Heffter array $H(n; k)$ is an integer Heffter array with $m=n$ and $s=t=k$.
In \cite{ADDY,DW} it is proved that

\begin{theorem}\label{Heffterwithemptycells}
There is an integer $H(n; k)$ if and only if $3\leq k\leq n$ and $nk\equiv 0,3 \pmod 4$.
\end{theorem}

We now define signed magic rectangles which are similar to Heffter arrays.
A {\em signed magic rectangle} $SMR(m,n;r, s)$ is an $m \times n$ array with entries from $X$, where
$X=\{0,\pm1,\pm2,\ldots, $ $\pm (mr-1)/2\}$ if $mr$ is odd and $X = \{\pm1,\pm2,\ldots,\pm mr/2\}$ if $mr$ is even,
such that precisely $r$ cells in every row and $s$ cells in every column are filled,
every integer from set $X$ appears exactly once in the array and
the sum of each row and of each column is zero.
By the definition, $mr=ns$, $r\leq n$ and $s \leq m$. If $r=n$ or $s=m$, then the rectangle has no empty cells.
We denote by $SMR(m,n)$ a signed magic rectangle $SMR(m,n;n,m)$.
In the case where $m = n$, we call the array a {\em signed magic square}.
Signed magic squares represent a type of magic square where each number from the set $X$
is used once. If $A$ is a Heffter array $H(m, n; s, t)$, then the array $[A \; -A]$ is a signed magic array $SMA(m,2n;2s,t)$, as noticed in \cite{KSW}, and the array
$\left[\begin{array}{cc}
A& \\
& -A \\
 \end{array}\right]$ is a signed magic array $SMG(2m,2n;s,t)$.

The following two theorems can be found in \cite{KSW}.
\begin{theorem} \label{TH:KSW1}
	An $SMR(m,n)$ exists precisely when $m = n = 1$, or when $m = 2$ and $n \equiv 0, 3 \pmod4$, or when $n = 2$ and $m \equiv 0, 3 \pmod4$, or when $m, n > 2$.
\end{theorem}

In \cite{KSW} the notation $SMS(n;k)$ is used for
a signed magic square with $k$ filled cells in each row and
$k$ filled cells in each column.

\begin{theorem}\label{TH:KSW2}
	There exists an $SMS(n;k)$ precisely when $n=k = 1$ or $3\leq k\leq n$.
\end{theorem}

The existence of an $SMR(m,n;r,s)$ is an open problem.
In \cite{AE} the authors study the signed magic rectangles with precisely two filled cells in each column. Below is the main theorem in
\cite{AE}.

\begin{theorem}\label{2filled cells}
There exists an $SMR(m,n;r,2)$ if and only if either $m=2$ and $n=r\equiv 0, 3 \pmod{4}$ or $m,r\geq 3$ and $mr=2n$.
\end{theorem}

\noindent In this paper, we prove that a signed magic rectangle
$SMR(m,n;$ $k,3)$ exists if and only if $3\leq m,k\leq n$ and $mk=3n$.

Consider an $SMR(m,n;k,3)$. By definition, we have $mk=3n$. So if $n$ is odd, then $m$ and $k$ must be odd.
In addition, the condition $mk=3n$ implies that either $3| k$ or $3| m$.
In Section \ref{SEC 2} we study the existence of an $SMR(m,n;k,3)$ when $n$ is odd and $3| k$.
In Section \ref{SEC 3} we study the existence of an $SMR(m,n;k,3)$ when $n$ is odd and $3| m$.
In Section \ref{SEC 4} we study the existence of an $SMR(m,n;k,3)$ when $n,k$ are even.
In Section \ref{SEC 5} we study the existence of an $SMR(m,n;k,3)$ when $n$ is even and $k$ is odd.

\section{The existence of an $SMR(m,n;k,3)$ when $n$ is odd and $3| k$.}\label{SEC 2}

The following result shows a relationship between an $MR(m,n;r,s)$ and an $SMR(m,n; r,s)$ when  $mr$ is odd.

\begin{lemma}\label{MStoSMR}
	If there exists an $MR(m,n; r,s)$ with $mr$ odd, then there exists an $SMR(m,n;r,s)$.
\end{lemma}

\begin{proof}
	
	Let $m, r$ be odd and greater than 3. By assumption, there exists an $MR(m,n; r,s)$, say $A$.
Since $mr=ns$, it follows that $n$ and $s$ are also odd.
	We will construct an $m\times n$ array $B$ as follows. The cell $(r,c)$ of $B$ is filled with $a-(mr-1)/2$ if and only if the cell $(r,c)$ of $A$ contains $a$.
As the entries in $A$ are precisely the integers $0$ through $mr-1$, it follows that the entries in $B$ are precisely the integers $-\frac{mr-1}{2}$ through $\frac{mr-1}{2}$, the required set of integers for an $SMR(m,n;r,s)$. It remains to be shown that $B$ has rows and columns summing to zero.

We know that the column sum of $A$ is $(mr-1)mr/(2n)$. So the column sum of $B$ is
$$\dfrac{(mr-1)mr}{2n}-\dfrac{(mr-1)s}{2}= \dfrac{(mr-1)}{2}(\dfrac{mr}{n}-s)=0.$$
Similarly, the row sum of $B$ is also zero.
	
\end{proof}

\begin{theorem}\cite{KL2} \label{Th.m.km.ks.s}
Let $k,m,s$ be positive integers. Then there exists a magic rectangle $MR(m,km;ks,s)$ if and only if
$m=s=k=1$ or $2\leq s\leq m$ and either $s$ is even or $km$ is odd.
\end{theorem}

By Lemma \ref{MStoSMR} and Theorem \ref{Th.m.km.ks.s} we obtain the following:

\begin{proposition}\label{nodd3divk}
 Let $k,m,s$ be positive odd integers. Then there exists an $SMR(m,km;ks,s)$ if
$m=s=k=1$ or $3\leq s\leq m$. In particular,
there exists an $SMR(m,km;3k,3)$.
\end{proposition}

\section{The existence of an $SMR(m,n;k,3)$ when $n$ is odd and $3| m$.}\label{SEC 3}

\begin{theorem}\label{am.bm.b.a} \cite{KL2}
Let $a, b, c$ be positive integers with $2 \leq a \leq b$. Let
$a, b, c$ be all odd, or let
$a$ and $b$ both be even, $c$ arbitrary, and $(a, b)\neq (2, 2)$.
Then there exists an $MR(ac,bc;b,a)$.
\end{theorem}

\begin{proposition}\label{nodd3divm}
If $n$ is odd and $3| m$, then there exists an $SMR(m,n;k,3)$.
\end{proposition}

\begin{proof}
If an $SMR(m,n;k,3)$ exists, then $mk=3n$. So by assumption $m$ and $k$ are odd.
Let $m=3c$. Then $mk=3n$ implies that $n=kc$. Now apply Theorem \ref{am.bm.b.a},
with $a=3$, $b=k$ and $c$ to obtain an $MR(m,n;k,3)$.  Now by Lemma \ref{MStoSMR}
there exists an $SMR(m,n;k,3)$.
\end{proof}

\section{The existence of an $SMR(m,n;k,3)$ when $n,k$ are even.}\label{SEC 4}

In this section and the next section we make use of the following result.
Since the structure of the $SMR(3,2)$ given below is crucial in our constructions of an $SMR(m,n;k,3)$ we include the proof of this lemma here which can also be found in \cite{KSW}.

\begin{lemma} \label{3xeven}
	An $SMR(3,n)$  exists if $n$ is even.
\end{lemma}

\begin{proof}
An $SMR(3,2)$ and an $SMR(3,4)$ are given in Figure \ref{3x2.3x4}.

\begin{figure}[ht]
$$\begin{array}{cccc}
    \begin{array}{|c|c|}
	\hline
	1 & -1 \\\hline
	2 & -2 \\\hline
	-3 & 3 \\\hline
	\end{array}&&&
	\begin{array}{|c|c|c|c|}\hline
	1 & -1 & 2 & -2 \\\hline
	5 & 4 & -5 &-4 \\\hline
	-6 & -3 & 3 & 6 \\\hline
	\end{array}\\
\end{array}$$
\caption{An $SMR(3,2)$ and an $SMR(3,4)$ }
		\label{3x2.3x4}
	\end{figure}

    Now let $n=2k\geq 6$ and $p_{j}=\lceil\frac{j}{2}\rceil$ for $1\leq j\leq 2k$ . Define a $3\times n$ array $A=[a_{i,j}]$ as follows: For $1\leq j\leq 2k$,

$$a_{1,j}= \begin{cases}
	- \left(\frac{3p_{j}-2}{2}\right) & j \equiv 0 \pmod4 \\
	\frac{3p_{j}-1}{2}& j \equiv 1 \pmod4 \\
	-\left(\frac{3p_{j}-1}{2}\right) & j \equiv 2 \pmod4 \\
	\frac{3p_{j}-2}{2} & j \equiv 3 \pmod4. \\
	
\end{cases}$$
For the third row we define $a_{3,1}=-3k$, $a_{3,2k}=3k$ and when $2\leq j\leq 2k-1$

$$a_{3,j}= \begin{cases}
	- 3(k-p_{j}) & j \equiv 0\pmod4 \\
	3(k-p_{j}+1) & j \equiv 1 \pmod4 \\
	-3(k-p_{j}) & j \equiv 2 \pmod4 \\
	3(k-p_{j}+1) & j \equiv 3 \pmod4. \\
\end{cases}$$
Finally, $a_{2,j}=-(a_{1,j}+a_{3,j})$ for $1\leq j\leq2k$ (see Figure \ref{3x10}).
It is straightforward to see that array $A$ is an $SMR(3,n)$.
\end{proof}

\begin{figure}[ht]
$$\begin{array}{|c|c|c|c|c|c|c|c|c|c|}\hline
1&-1&2&-2&4&-4&5&-5&7&-7 \\\hline
14&13&-14&11&-13&10&-11&8&-10&-8	\\\hline
-15&-12&12&-9&9&-6&6&-3&3&15\\\hline
\end{array}$$
\caption{An $SMR(3,10)$ using the method given in Lemma \ref{3xeven}.}
		\label{3x10}
\end{figure}

Let ${\cal A}=\{A_1,A_2,\ldots,A_r\}$ and ${\cal B}=\{B_1,B_2,\ldots,$ $B_s\}$ be two partitions of a set $S$.
We say the partitions ${\cal A}$ and ${\cal B}$ are {\em near orthogonal} if $|A_i\cap B_j|\leq 1$
for all $1\leq i\leq r$ and $1\leq j\leq s$.

\begin{theorem}\label{mainconstruction}
Let $A$ be an $SMR(m,n)$ with entries from the set $X$. Let ${\cal P}_1=\{C_1,C_2,\ldots,C_n\}$, where
$C_i$'s are the columns of $A$.
Let $k\geq m$ and $k|mn$. If there exists a partition ${\cal P}_2=\{D_1,D_2,\ldots,D_{\ell}\}$ of $X$, where  $\ell=mn/k$,
such that $|D_i|=k$, the sum of members in each $D_i$ is zero for $1\leq i\leq \ell$, and  ${\cal P}_1$ and ${\cal P}_2$
are near orthogonal, then there exists an $SMR(mn/k,n;k,m)$.
\end{theorem}

\begin{proof}
Let $B$ be an $mn/k$ by $n$ empty array. We use the members of $D_i$ to fill $\ell$ cells of row $i$ of $B$. Let $d\in D_i$. Then there is a unique
column $C_j$ of $A$ which contains $d$. We place $d$ in row $i$ and column $j$ of $B$. Note that the members used in $B$ are
precisely the members in $X$ because ${\cal P}_2$ is a partition of $X$. Since ${\cal P}_1$ and ${\cal P}_2$ are near orthogonal, each cell of $B$ has at most one member.
By construction, row $i$ of $B$ and $D_i$ have the same members and column $j$ of $B$ and $C_j$ have the same members.
So the sum of each row and each column of $B$ is zero. Hence, $B$ is an $SMR(mn/k,n;k,m)$.
\end{proof}

\begin{proposition}\label{n and k even}
Let $k$ and $n$ be even integers, $k\geq 4$ and $k|3n$. Then there exists an $SMR(3n/k,n;k,3)$.
\end{proposition}

\begin{proof}
Let $A$ be the $SMR(3,n)$ constructed in the proof of Lemma \ref{3xeven} with elements in
$X=\{\pm 1, \pm 2,\ldots,\pm 3n/2\}$.
Let ${\cal P}_1=\{C_1,C_2,\ldots,C_n\}$, where
$C_i$'s are the columns of $A$. Obviously, ${\cal P}_1$ is a partition of $X$.
By the proof of Lemma \ref{3xeven}, we see that if $x$ appears in row $i$ of $A$, then $-x$
also appears in row $i$.

We construct a partition ${\cal P}_2=\{D_1,D_2,\ldots,D_{\ell}\}$ of $X$, where  $\ell=3n/k$,
such that $|D_i|=k$, the sum of members in each $D_i$ is zero for $1\leq i\leq \ell$, and  ${\cal P}_1$ and ${\cal P}_2$
are near orthogonal. Then by Theorem \ref{mainconstruction}, there exists an $SMR(3n/k,n;k,3)$, as desired.
Let $n=kq+r$ with $0\leq r<k$. By assumption, it is easy to see that $r=0$, $r=k/3$ or $r=2k/3$.
So we consider 3 cases.

\vspace{5mm}

\noindent {\bf Case 1: $r=0$.}\quad Since $k$ and $n$ are even and for $x\in X$ both $x$ and $-x$ appear in the same row of $A$, we can partition each row of $A$ into $n/k$ $k$-subsets such that if $x$ is in a $k$-subset, then $-x$ is also in that $k$-subset. Hence the sum of members in each $k$-subset is zero.
Let ${\cal P}_2=\{D_1,D_2,\ldots,D_{\ell}\}$, $\ell=3n/k$, be the collection of all these $k$-subsets.
Since $D_i$ is a subset of a row of $A$ for $1\leq i\leq \ell$, it follows that the partitions ${\cal P}_1$ and ${\cal P}_2$ are near orthogonal. So the result follows by Theorem \ref{mainconstruction}.

\vspace{5mm}

\noindent {\bf Case 2: $r=k/3$.}\quad Then $\ell=3n/k=3q+1$. We partition each row of $A$ into $q$ $k$-subsets and one $k/3$-subset.
Note that $k/3$ is an even number because $k$ is even. First we form a $k/3$-subset $E_i$ of row $i$ for $i=1,2,3$, such that if $x$ is in $E_i$, then $-x$ is also in $E_i$. In addition, no two members of $D_{\ell}=E_1\cup E_2\cup E_3$ are in the same column of $A$. Now consider the set $F_i$ which consists of the elements in row $i$ that are not in $E_i$ for $i=1,2,3$. The size of $F_i$ is even and the members of $F_i$ can be paired as $x,-x$ for some $x\in F_i$.
Hence we can partition each $F_i$ into $k$-subsets  $D_1^i, D_2^i,\ldots,D_q^i$ such that if $a\in D_j^i$, then $-a\in D_j^i$, where $1\leq i\leq 3$ and $1\leq j\leq q$.
Consider the partition ${\cal P}_2=\{D_1^i,D_2^i,\ldots,D_q^i, D_{\ell}\mid 1\leq i\leq 3\}$.
By construction, no two members of a $k$-subset $Y\in {\cal P}_2$ belong to the same column of $A$,
so the partitions ${\cal P}_1$ and ${\cal P}_2$ are near orthogonal.
In addition, the sum of the members of $Y$ is zero.
Now the result follows by Theorem \ref{mainconstruction}.

\vspace{5mm}

\noindent {\bf Case 3: $r=2k/3$.}\quad Then $\ell=3n/k=3q+2$. We partition each row of $A$ into $q$ $k$-subsets and two $k/3$-subsets.
Note that $k/3$ is an even number because $k$ is even. First we form two $k/3$-subsets $E_i^1, E_i^2$ of row $i$ for $i=1,2,3$, such that if $x$ is in $E_i^1$ or $E_i^2$, then $-x$ is also in $E_i^1$ or $E_i^2$, respectively. In addition, no two members of $D_{\ell-1}=E_1^1\cup E_2^1\cup E_3^1$ or of  $D_{\ell}=E_1^2\cup E_2^2\cup E_3^2$
are in the same column of $A$. Now consider the set $F_i$ which consists of the elements in row $i$ that are not in $E_i^1\cup E_i^2$ for $i=1,2,3$. The size of $F_i$ is even and the members of $F_i$ can be paired as $x,-x$ for some $x$.
Hence, we can partition each $F_i$ into $k$-subsets  $D_1^i, D_2^i,\ldots,D_q^i$ such that if $a\in D_j^i$, then $-a\in D_j^i$, where $1\leq i\leq 3$ and $1\leq j\leq q$.
Consider the partition ${\cal P}_2=\{D_1^i,D_2^i,\ldots,D_q^i, D_{\ell-1}, D_{\ell}\mid 1\leq i\leq 3\}$.
By construction, no two members of a $k$-subset $Y\in {\cal P}_2$ belong to the same column of $A$, so
the partitions ${\cal P}_1$ and ${\cal P}_2$ are near orthogonal.
In addition, the sum of the members of $Y$ is zero.
Now the result follows by Theorem \ref{mainconstruction}.
\end{proof}

\section{The existence of an $SMR(m,n;k,3)$ when $n$ is even and $k$ is odd.}\label{SEC 5}
Let $n$ be a positive integer.
Let $A$ be the $SMR(3,n)$ constructed in the proof of Lemma \ref{3xeven} with elements in
$X=\{\pm 1, \pm 2,\ldots,$ $\pm 3n/2\}$.
For $i=1,2,3$, let $R_i$ consist of the entries in row $i$ of $A$.
By the proof of Lemma \ref{3xeven}, if $n\equiv 0 \pmod{4}$, then
\begin{equation}\label{R1R2.n=0mod4}
\begin{array}{l}
R_1=\{\pm (3i+1), \pm(3i+2) \mid 0\leq i\leq (n-4)/4\},\\
R_2=\{\pm (3i+1), \pm(3i+2) \mid n/4\leq i\leq (n-2)/2\}.
\end{array}
\end{equation}

\noindent If $n\equiv 2 \pmod{4}$, then
\begin{equation}\label{R1R2.n=2mod4}
\begin{array}{lll}
R_1&=&\{\pm (3i+1) \mid 0\leq i\leq (n-2)/4\}\\
&\cup&\{ \pm(3i+2)\mid 0\leq i\leq (n-6)/4\},\\
R_2&=&\{\pm (3i+1) \mid (n+2)/4\leq i\leq (n-2)/2\}\\
& \cup&\{\pm(3i+2)\mid (n-2)/4\leq i\leq (n-2)/2\}.
\end{array}
\end{equation}
And
\begin{equation}\label{R3.n=}
R_3=\{\pm3i\mid 1\leq i\leq n/2\}.
\end{equation}

For a set of numbers $L$ we define $L'=\{-a\mid a\in L\}$.

\begin{remark}\label{rmk}
Let $n>k$, $k$ odd and $n$ even. If $k| 3n$ and $n=kq+r$, where $0\leq r <k$, then
$r=0, k/3$ or $2k/3$.
\end{remark}

\begin{lemma}\label{Lem:S_1.S_2}
Let $n>k\geq 5$ with $k$  odd and $n$ even such that $k| 3n$. Let $n=kq+r$, where $0\leq r<k$.
Then there exist two sets of {\rm 3}-subsets of $X=\{\pm 1, \pm 2,\ldots,\pm3n/2\}$, say $S_1$ and $S_2$,
such that
\begin{enumerate}
\item the member sum of each {\rm 3}-subset is zero;
\item each {\rm 3}-subset of $S_1$ has one member in $R_1$ and two members in $R_2$
and each {\rm 3}-subset of $S_2$ has one member in $R_2$ and two members in $R_1$;
\item if $\{a,b,c\}$ is a member of $S_1$ or of $S_2$, then $\{-a,-b,-c\}$ is also a member of $S_1$ or $S_2$, respectively;
\item $|S_1|=|S_2|\geq q,q+1,q+2$ if $r=0,k/3, 2k/3$, respectively;
\item the {\rm 3}-subsets of $S_1\cup S_2$ are all disjoint.
\end{enumerate}
\end{lemma}

\begin{proof}
Figure \ref{small cases} displays the sets $S_1=\{M_i,M_i'\}$ and $S_2=\{N_i,N_i'\}$ for
$(n,k)\in \{(10,5),(12,9),(14,7),(18,9),(20,5),\newline (20,15),(22,11),(24,9), (26,13),
(28,7),(28,14),(30,5), (30,9),\newline(30,15)\}.$
It is easy to see that the 3-sets in $S_1$ and $S_2$ satisfy items 1-5.
So, we may assume from now on that $n \geq 34$.

Let $p=\lceil q/4\rceil$ and note that $q=\lfloor n/k\rfloor \leq n/k$, hence
\begin{equation}\label{EqforP}
p=\lceil q/4\rceil\leq \lceil n/(4k)\rceil < n/(4k)+1.
\end{equation}
For $0\leq i \leq p-1$ we define
\begin{equation}\label{Eq5}
\begin{array}{l}
M_{i_1}=\{3i+1,	3n/2-2-12i,	-(3n/2-1-9i)\}\mbox{ and }\\
M_{i_2}=\{3i+2,	3n/2-7-9i,	-(3n/2-5-6i)\}.
\end{array}
\end{equation}
By construction, the member sums of $M_{i_1}$ and of  $M_{i_2}$ are zero.

We also note that if $n\equiv 0 \pmod 4$, by (\ref{EqforP}),
$$3i+1, 3i+2\leq 3(n-4)/4+2$$
and

$$3n/2-2-12i\geq 3n/2-2-12(p-1)> 3n/2-3n/k-2.$$
So $3n/2-2-12i>3(n-4)/4+2.$
Similarly,
$$
3n/2-1-9i,
3n/2-7-9i,
3n/2-5-6i
> 3(n-4)/4+2,$$
for $0\leq i\leq p-1$.

If $n\equiv 2 \pmod 4$, by (\ref{EqforP}),
$$3i+1,3i+2\leq 3(n-2)/4+1$$
and
$$3n/2-2-12i> 3n/2-3n/k -2 > 3(n-2)/4+1.$$

Similarly,
$$3n/2-1-9i, 3n/2-7-9i,3n/2-5-6i > 3(n-2)/4+1,$$
for $0\leq i \leq p-1$.

Hence, $M_{i_1}$ and $M_{i_2}$ each have one member in $R_1$ and two members in $R_2$.

If $3n/2-2-12i=3n/2-5-6j$, then $-4i+2j=-1$ which is impossible.

If $3n/2-7-9i=3n/2-1-9j$, then $3(j-i)=2$ which is impossible.

Therefore the $2p$ 3-subsets $M_{i_1}$ and $M_{i_2}$ are disjoint and if $a$ appears in this collection of
3-subsets, $-a$ does not appear in this collection.

We now define $2p$ 3-subsets each consisting of two members in $R_1$ and one member in $R_2$.
First let $n\equiv 0\pmod 4$. For $0\leq i\leq p-1$ define
\begin{equation}\label{Eq6n=0}
\begin{array}{l}
N_{i_1}=\{3i+3p+1,(3n/4)-2-6i , -[(3n/4)-1+3(p-i)]\}\\
N_{i_2}=\{3i+3p+2,(3n/4)-4-6i , -[(3n/4)-2+3(p-i)]\}.
\end{array}
\end{equation}
By construction, the member sums of $N_{i_1}$ and of  $N_{i_2}$ are zero.
We also note that $3p+1+3i,3n/4-2-6i, 3p+2+3i,3n/4-4-6i \leq 3(n-4)/4+2$  and $3n/4-1+3(p-i), 3n/4-2+3(p-i) > 3(n-4)/4+2$,
where $0\leq i\leq p-1$.

Hence, $N_{i_1}$ and $N_{i_2}$ each have two members in $R_1$ and one member in $R_2$.

If $3p+1+3i=3n/4-2-6j$, then $n/4=p+i+2j+1\leq 4p-2 < n/k+2$,  by (\ref{EqforP}).
If $k\geq 7$, then $n/4<n/7+2$ which is false because $n\geq 34$.
If $k=5$, then $n/4<n/5+2$ is false for $n\geq 40$.

If $3p+2+3i=3n/4-4-6j$, then $n/4=p+i+2j+2\leq 4p-1<n/k+3$,  by (\ref{EqforP}).
If $k\geq 7$, then $3p+2+3i=3n/4-4-6i$, then $n/4\leq 4p-1<n/4+3$.
If $k\geq 7$, then $n/4<n/7+3$ is false because $n\geq 34$.
If $k=5$, then $n/4<n/5+3$ is false for $n\geq 60$. The remaining case is $(n,k)=(40,5)$ in which case $p=2$, giving the contradiction $10<4p-1=7$.

Therefore the $2p$ 3-subsets $N_{i_1}$ and $N_{i_2}$ are disjoint, and if $a$ appears in this collection of
3-subsets, $-a$ does not appear in this collection.


Second, let $n\equiv 2\pmod 4$. For $0\leq i\leq p-1$,
define
{\small
\begin{equation}\label{Eq7n=2}
\begin{array}{l}
N_{i_1}=\{3i+3p+1,(3n-2)/4-6i , -[(3n-2)/4+1+3(p-i)]\}\\
N_{i_2}=\{3i+3p+2,(3n-2)/4-2-6i , -[(3n-2)/4+3(p-i)]\}.
\end{array}
\end{equation}
}
By construction, the member sums of $N_{i_1}$ and of  $N_{i_2}$ are zero.

We also note that $3p+1+3i,(3n-2)/4-6i, 3p+2+3i,(3n-2)/4-2-6i \leq 3(n-2)/4+1$  and ($3n-2)/4+1+3(p-i), (3n-2)/4+3(p-i) > 3(n-2)/4+1$,
where $0\leq i\leq p-1$.

Hence, $N_{i_1}$ and $N_{i_2}$ each have two members in $R_1$ and one member in $R_2$.

If $3p+1+3i=(3n-2)/4-6j$, then $(n+2)/4=p+1+i+2j\leq 4p-2<n/k+2$, by  (\ref{EqforP}), which is impossible because $n\geq 34$.

If $3p+2+3i=(3n-2)/4-2-6j$, then $(n+2)/4=p+2+i+2j\leq 4p-1<n/k+3$,  by (\ref{EqforP}).
If $k\geq 7$, then $(n+2)/4<n/7+3$ is false because $n\geq 34$, If $k=5$, then $(n+2)/4<n/5+3$ is false for $n\geq 50$.
(There are no exceptional cases, since $n\geq 34$, $n\equiv 2\pmod4$ and $5|n$ imply $n\geq 50$.)

Therefore the $2p$ 3-subsets $N_{i_1}$ and $N_{i_2}$ are disjoint and if $a$ appears in this collection of
3-subsets, $-a$ does not appear in this collection.

We now prove that no member of $X$ appears more than once in the $4p$ 3-subsets constructed above.
First we prove this claim when $n\equiv 0 \pmod 4$.
Consider the 3-subsets given in (\ref{Eq5}) and (\ref{Eq6n=0}).

Obviously, $3i+1\neq 3p+1+3j $ and $3i+2\neq 3p+2+3j$ because $p\geq 1$.

If $3i+1=(3n/4)-2-6j$, then $n/4=i+2j+1\leq 3p-2<3n/4k+1$,  by (\ref{EqforP}), which is impossible because
$n\geq 34$.

If $3i+2=3n/4-4-6j$, then $n/4=i+2j+2\leq 3p-1<3n/4k+2$,  by (\ref{EqforP}), which is also impossible because
$n\geq 34$.

If $3n/2-2-12i=3n/4-2+3p-3j$, then $n/4=p+4i-j\leq 5p-4<5n/4k+1$, which is impossible because $n\geq 34$.

If $3n/2-7-9i=3n/4-1+3p-3j$, then $n/4=p+3i-j+2\leq 4p-1< n/k+3$, which is impossible because $n\geq 34$.


If $3n/2-1-9i=3n/4-1+3p-3j$, then $n/4=p+3i-j\leq 4p-3<n/k+1$,  by (\ref{EqforP}), which is impossible because $n\geq 34$.

If $3n/2-5-6i=3n/4-2+3p-3j$, then $n/4=2i-j+p+1\leq 3p-1<3n/4k+2$, by (\ref{EqforP}), which is not possible if $n\geq 34$.


Second, we prove the claim above when $n\equiv 2 \pmod 4$.

Consider the 3-subsets given in (\ref{Eq5}) and (\ref{Eq7n=2}).

If $3i+1=(3n-2)/4-6j$, then $(n+2)/4=i+2j+1\leq 3p-2<3n/4k+1$, which is impossible because $n\geq 34$.

If $3i+2=(3n-2)/4-2-6j$, then $(n+2)/4 =i+2j+2\leq 3p-1<3n/4k+2$, which is impossible because $n\geq 34$.

If $3n/2-2-12i=(3n-2)/4+3p-3j$, then $(n+2)/4=p-j+4i+1\leq 5p-3<5n/4k+2$, which is impossible for $n\neq 10(2a+1)$,
where $a$ is a nonnegative integer. If $n=10(2a+1)$, then the inequality $(n+2)/4\leq 5p-3$ does not hold.

If $3n/2-1-9i=(3n-2)/4+1+3(p-j)$, then $(n+2)/4=3i+p-j+1\leq 4p-2<n/k+2$, which is impossible because $n\geq 34$.

If $3n/2-7-9i=(3n-2)/4+1+3p-3j$, then $(n+2)/4=3i+p-j+3\leq 4p < n/k+4$, which is impossible because $n\geq 34$.

If $3n/2-5-6i=(3n-2)/4+3(p-j)$, then $(n+2)/4=p-j+2i+2<3p<3n/4k+3$, which is impossible because $n\geq 34$.

 We are now ready to construct the sets $S_1$ and $S_2$. Recall that for a set of numbers $L$ we defined $L'=\{-a\mid a\in L\}$.
 When $n=kq$, $n=kq+k/3$ or $n=kq+2k/3$ and $q\equiv 2\pmod{4}$
 the desired $S_1$ and $S_2$ are given by:
\begin{equation}\label{Eq8}
\begin{array}{l}
S_1=\{M_{i_1},M_{i_2},M_{i_1}', M_{i_2}'\mid 0\leq i\leq p-1\}\\
S_2=\{N_{i_1},N_{i_2},N_{i_1}', N_{i_2}'\mid 0\leq i\leq p-1\}.
\end{array}
\end{equation}
Then $|S_1|+|S_2|=4p$.
Now if $n=kq$, then $4p\geq q$, if $n=kq+k/3$, then $4p\geq q+1$ because $q$ is odd,
if $n=kq+2k/3$ and $q \equiv 2\pmod{4}$ then $4p=q+2$ and
if $n=kq+2k/3$ and $q \equiv 0\pmod{4}$ then $4p=q$.
We now add two 3-subsets to $S_1$ and two 3-subsets to $S_2$ to obtain $|S_1|=|S_2|=q+2$.
Define
$${\bar M_{i_1}}=M_{i_1}\cup\{\{6p+1, \frac{3n-2}{4}+3p+3 , -(\frac{3n-2}{4}+9p+4) \}\},$$
$${\bar N_{i_1}}=N_{i_1}\cup\{\{9p+1, \frac{3n-2}{4}-6p+3, -(\frac{3n-2}{4}+3p+4) \}\}$$
and

\begin{equation}\label{Eq8.1}
\begin{array}{l}
S_1=\{{\bar M_{i_1}},M_{i_2},{\bar M_{i_1}'}, M_{i_2}'\mid 0\leq i\leq p-1\}\\
S_2=\{{\bar N_{i_1}},N_{i_1},{\bar N_{i_1}'}, N_{i_2}'\mid 0\leq i\leq p-1\}.
\end{array}
\end{equation}
It is straightforward to see that $S_1$ and $S_2$ satisfy the required conditions.
\end{proof}

\begin{lemma}\label{Lem:S_3}
 Let $n>k\geq 5$ with $k$  odd and $n$ even such that $k| 3n$. Let $n=kq+r$, where $0\leq r<k$, and
$$R_3=\{\pm3i\mid 1\leq i\leq n/2\}.$$
Then there is a set $S_3$ of disjoint {\rm 3}-subsets of $R_3$ such that
\begin{enumerate}
\item the member sum of each {\rm 3}-subset is zero;
\item if $\{a,b,c\}\in S_3$, then $\{-a,-b,-c\}\in S_3$;
\item $|S_3|\geq q, q+1, q+2$ if $r=0, k/3, 2k/3$, respectively.
\end{enumerate}
\end{lemma}

\begin{proof}
If $n=2k$ or $(n,k)\in\{(12,9),(20,15),(28,21)\}$ define $S_3=\{\{3,6,-9\}, \{-3,-6,9\}\}$.

For $(n,k)\in \{(20,5),(24, 9), (28,7), (30,9)\}$ define
$S_3=\{\{3,27,$ $-30\}$, $\{-3,-27,30\}$, $\{6,12,-18\}$, $\{-6,-12,18\}\}$.

For $(n,k)=(30,5)$ define
$S_3=\{\{3,42,-45\}$, $\{6,33,-39\}$, $\{9,21,-30\}$, $\{-3,-42,45\}$, $\{-6,-33,39\}$, $\{-9,-21,30\}\}$.
Hence, the statement is true for $n\leq 30$.
For $n>30$ we proceed as follows:
Let $\alpha=\lfloor(n-8)/12\rfloor$.
Define
$$T_{i_1}=\{3+6i, 6\alpha+9+6i, -(6\alpha+12+12i)\},$$
for $0\leq i\leq \alpha$. Since
$$ 12+6\alpha+12i\leq 12+18\alpha\leq 12+18(n-8)/12=3n/2,$$
it follows that $T_{i_1}, T_{i_1}'\subseteq R_3$.

Now define
$$T_{i_2}=\{12\alpha+15+6i,6\alpha-6-12i , -(18\alpha+9-6i)\},$$
for $0\leq i \leq \lfloor(\alpha-2)/2\rfloor$. See Figure \ref{3-subsestinR3} for illustration.
Note that $6\alpha-6-12\lfloor(\alpha-2)/2\rfloor\geq 6$ and
$18\alpha+9\leq 3n/2$. Hence, $T_{i_2}, T_{i_2}'\subseteq R_3$. In addition, the absolute value of the odd numbers
used in $T_{i_1}$ are $\{3, 9,\ldots,12\alpha+9\}$ and used in $T_{i_2}$ are $\{12\alpha+15, 12\alpha+21,\ldots,18\alpha+9\}$. The absolute value of the even numbers used in $T_{i_1}$ are $\{6\alpha+12+12i\mid 0\leq i\leq \alpha\}$ and used in $T_{i_2}$ are $\{6\alpha-6-12i\mid 0\leq i\leq\lfloor (\alpha-2)/2\rfloor\}$.
Hence, the numbers used in $T_{i_1}$ and $T_{i_2}$ are all different.
Define
$$S_3=\{T_{i_1},T_{i_1}'\mid 0\leq i\leq \alpha\}\cup
 \{T_{i_2},T_{i_2}'\mid 0\leq i\leq \lfloor (\alpha-2)/2\rfloor\}.$$
It is straightforward to confirm that $|S_3|\geq q, q+1, q+2$ if $r=0, k/3, 2k/3$, respectively.
\end{proof}

\begin{figure}[ht]
$$\begin{array}{cc}
\begin{array}{|c|c|}\hline
\{3,45,-48\}&\{-3,-45,48\}\\ \hline
\{9,51,-60\}&\{-9,-51,60\} \\\hline
\{15,57,-72\}&\{-15,-57,72\}\\ \hline
\{21,63,-84\}&\{-21,-63,84\}\\ \hline
\{27,69,-96\}&\{-27,-69,96\}\\ \hline
\{33,75,-108\}&\{-33,-75,108\}\\ \hline
\{39,81,-120\}&\{-39,-81,120\}\\ \hline
\end{array}&
\begin{array}{|c|}\hline
\{87,30,-117\}\\ \hline
\{93,18,-111\}\\ \hline
\{99,6,-105\}\\ \hline
 \{-87,-30,117\}\\ \hline
 \{-93,-18,111\} \\ \hline
\{-99,-6,105\}\\ \hline
\end{array}\\
L_{i_1} \mbox{ and } L_{i_1}' & L_{i_2} \mbox{ and } L_{i_2}'\\
\end{array}$$
\caption{\mbox {20 disjoint {\rm 3}-subsets in } $R_3$ \mbox { when } $(n,k)=(90,5)$}
\label{3-subsestinR3}
\end{figure}

\begin{proposition}\label{n even and k odd}
Let $k$ be odd and $n$ be even integers, $k\geq 3$ and $k|3n$. Then there exists an $SMR(3n/k,n;k,3)$.
\end{proposition}

\begin{proof}
In \cite{KSW} it is proved that there is an $SMR(n,n;3,3)$ for $n\geq 3$. Hence, the statement is true
for $k=3$. Now let $k\geq 5$.
Let $A$ be the $SMR(3,n)$ constructed in the proof of
Lemma \ref{3xeven} with entries $X=\{\pm 1, \pm 2,\ldots,\pm 3n/2\}$.
Let ${\cal P}_1=\{C_1,C_2,$ $\ldots,C_n\}$, where
$C_i$'s are the columns of $A$. Obviously, ${\cal P}_1$ is a partition of $X$.
As in the proof of Proposition \ref{n and k even}, we construct a partition
${\cal P}_2=\{D_1,D_2,\ldots,D_{\ell}\}$ of $X$, where  $\ell=3n/k$,
such that $|D_i|=k$, the sum of members in each $D_i$ is zero for $1\leq i\leq \ell$, and  ${\cal P}_1$ and ${\cal P}_2$
are near orthogonal. Then by Theorem \ref{mainconstruction}, the result follows.

Since $n=kq+r$, where $0\leq r<k$, it follows that $r=0$, $r=k/3$ or $r=2k/3$ by Remark \ref{rmk}.
We consider 3 cases.
\vspace{5mm}

\noindent {\bf Case 1: $r=0$.}\quad So $n=kq$. Since $k$ is odd and $n$ is even, it follows that $q$ is even.
By Lemma \ref{Lem:S_1.S_2},
there are sets $S_1$ and $S_2$ each consisting of $q$ 3-subsets of $R_1\cup R_2$ having the properties given in this lemma.
Consider the set $L_1=\{x\in R_1\mid x \mbox{ is not in any 3-subset of } S_1\cup S_2\}$. Then $|L_1|=n-3q=(k-3)q$.
Note that $k-3$ is even and the members of $L_1$ are of the form $\pm x$ for some $x\in R_1$.
Now we assign each 3-subset of $S_2$ to a $(k-3)$-subset of $L_1$ as follows: Let $\{a_i,b_i,c_i\}\in S_2$ with $a_i\in R_2$ and $b_i,c_i\in R_1$, where $1\leq i\leq q$.
Let $d_i$ be the member of $R_1$ which is in the same column as of $a_i$.
We partition the set $L_1$ into $q$ $(k-3)$-subsets of $L_1$ such that if $x$ is in a $(k-3)$-subset, then $-x$ is also in this subset. In addition,
the $i$th $(k-3)$-subset misses $d_i$. It is obvious that a $d_i$ can be in at most one of the $q$ $(k-3)$-subsets.
Let $D_i$ be the union of the $i$th $(k-3)$-subset and $\{a_i,b_i,c_i\}$, where $1\leq i\leq q$. Note that, by construction,
the $k$-subsets $D_i$ are disjoint, the member sums of $D_i$ are zero and intersect each column of $A$ in at most one member for $1\leq i\leq q$.

Now consider $L_2=\{x\in R_2\mid x \mbox{ is not in any {\rm 3}-subset of } S_1\cup S_2\}$ and the 3-subsets in $S_1$. By a similar method
described above, we find $k$-subsets $D_{q+1},D_{q+2},\ldots,D_{2q}$ which are disjoint, the member sums of $D_i$ are zero and
intersect each column of $A$ in at most one member for $q+1\leq i\leq 2q$.
In addition, the set $\{D_1,D_2,\ldots, D_{2q}\}$ partitions $R_1\cup R_2$.

Finally, consider the 3-subsets in $S_3$ given in Lemma \ref{Lem:S_3} and
the subset $L_3$ of $R_3$ which consists of members of $R_3$ which are not in any member of $S_3$.
Then $L_3$ can be partitioned into $q$ subsets of size $k-3$, say $Q_i$, where $1\leq i\leq q$, such that if $x\in Q_i$, then $-x\in Q_i$. Pair $q$ members of $S_3$ with $Q_i$'s to obtain $k$-subset $D_i$ for $2q+1\leq i\leq 3q$.
Then the sum of members of each $D_i$ is zero. In addition, $\{D_i\mid 2q+1\leq i\leq 3q\}$ is a partition of $R_3$. Hence, it has at most one member in common with each column of $A$.

Define ${\cal P}_2=\{D_1,D_2,\ldots,D_{\ell}\}$, where  $\ell=3n/k=3q$. Then
${\cal P}_2$ is a partition of $X$ and is near orthogonal to the partition
${\cal P}_1=\{C_1,C_2,\ldots,C_n\}$.
So by Theorem \ref{mainconstruction}, there exists an $SMR(3n/k,n;k,3)$.


\vspace{5mm}

\noindent {\bf Case 2: $r=k/3$.}\quad So $n=kq+k/3$. Since $n$ is even and $k$ is odd, it follows that $q$ is odd.
Figure \ref {4,12;9,3} displays an $SMR(4,12; 9, 3)$. In what follows, we assume $(n,k)\neq(12,9)$.
This case does not follow the construction given below.
By Lemma \ref{Lem:S_1.S_2},
there are sets $S_1$ and $S_2$, each containing $q+1$ 3-subsets of $R_1\cup R_2$ having the properties given in this lemma.
In what follows, we partition $R_1\cup R_2$ into $2q$ $k$-subsets and two $k/3$-subset such that the member sum of each subset is zero.

Consider the set $$L_1=\{x\in R_1\mid x \mbox{ is not in any {\rm 3}-subsets of } S_1\cup S_2\}$$ and
$$L_2=\{x\in R_2\mid x \mbox{ is not in any {\rm 3}-subset of } S_1\cup S_2\}.$$

Also consider the 3-subsets in $S_3$ given in Lemma \ref{Lem:S_3} and
the subset $L_3$ of $R_3$ which consists of members of $R_3$ which are not in any member of $S_3$.
Then
$$|L_1|=|L_2|= |L_3|=n-3(q+1)=(k-3)q+(k/3-3).$$


It is easy to see that there exist 3-subsets $E_i$ in $S_i$, $1\leq i\leq 3$, such that the subset
$E_1\cup E_2\cup E_3$ intersect each column of $A$ in at most one member.
Using the method described in Case 1 for construction of $(k-3)$-subsets of $L_1$ we construct $(k/3-3)$-subsets ${E'}_i$ of $L_i$ for $i=1,2,3$ such that if $x\in {E'}_i$, then $-x\in {E'}_i$.
In addition, no two members of
$$D_{\ell}=\bigcup_{i=1}^3 ({E'}_i\cup E_i), \mbox { where } \ell=3q+1 $$
are in the same column of $A$. By construction
the member sum of $D_{\ell}$ is zero. Note that $(k/3-3)\geq 0$ because $k\nmid n$ and $k$ is odd.

Consider the set $F_i=L_i\setminus {E'}_i$, $i=1,2,3$. Note that $|F_i|=(k-3)q$ which is even and the members of $F_i$ can be paired as $x,-x$ for some $x\in F_i$.
Now we pair the $q$ $(k-3)$-subsets of $F_1$ with members of $S_2\setminus\{ E_2\}$ to obtain $k$-subsets, say $D_i^1$, $1\leq i\leq q$ as described in Case 1.
We also pair the $q$ $(k-3)$-subsets of $F_2$ with members of $S_1\setminus\{E_1\}$ to obtain $k$-subsets, say $D_i^2$,
$1\leq i\leq q$.

Finally, we pair the $q$ $(k-3)$-subsets of $L_3$ with members of $S_3\setminus \{E_3\}$, to obtain $k$-subsets, say $D_i^3$, $1\leq i\leq q$. Note that the member sum of each of $k$-subset is zero.
By construction,
$${\cal P}_2=\{D_1^i,D_2^i,\ldots,D_q^i, D_{\ell}\mid 1\leq i\leq 3\}$$
is a partition of $X$  and is near orthogonal to the partition
${\cal P}_1=\{C_1,C_2,\ldots,C_n\}$.
So by Theorem \ref{mainconstruction}, there exists an $SMR(3n/k,n;k,3)$.
See Example \ref {example 19} for illustration.

\vspace{5mm}


\noindent {\bf Case 3: $r=2k/3$.}\quad So $n=kq+2k/3$. Since $n$ and $2k/3$ are even and $k$ is odd,
it follows that $q$ is even.
By Lemma \ref{Lem:S_1.S_2},
there are sets $S_1$ and $S_2$ each containing $q+2$ 3-subsets of $R_1\cup R_2$ having the properties given in this lemma.
We partition $R_1\cup R_2$ into $2q$ $k$-subsets and four $k/3$-subset as follows:
Let $L_1,L_2$ and $L_3$ be as defined in Case 2. Then
$$|L_1|=|L_2|= |L_3|=n-3(q+2)=(k-3)q+2(k/3-3).$$

It is easy to find two disjoint sets of 3-subsets $E_{i}^1$ and $E_{i}^2$ in $S_i$, $1\leq i\leq 3$, such that the subset
$E_{1}^1\cup E_{2}^1\cup E_{3}^1$ and $E_{1}^2\cup E_{2}^2\cup E_{3}^2$ intersect each column of $A$ in at most one member.

Now we construct two disjoint sets of $(k/3-3$)-subset ${E'}_{i}^1$ and ${E'}_{i}^2$
of $L_i$ for $i=1,2,3$ such that if $x\in {E'}_{i}^1$ or $x\in {E'}_{i}^2$, then
$-x\in {E'}_{i}^1$ or $-x\in {E'}_{i}^2$, respectively.
In addition, no two members of
$$D_{\ell-1}=\bigcup_{i=1}^3 ({E'}_{i}^1\cup E_{i}^1)
\mbox{ and }
D_{\ell}=\bigcup_{i=1}^3 ({E'}_{i}^2\cup E_{i}^2),
\mbox{ where }\ell=3q+2,
$$ are in the same column of $A$.
By construction,
the member sums  of $D_{\ell-1}$ and of $D_{\ell}$ are both zero.

Now we construct two $(k/3-3)$-subsets $E_{i}^1$ and $E_{i}^2$ of $L_i$ for $i=1,2,3$ such that if $x\in E_{i}^1$ or $x\in E_{i}^2$ ,
then $-x\in E_{i}^1$ or $-x\in E_{i}^2$, respectively. In addition, no two members of $E_{1}^1\cup E_{2}^1\cup E_{3}^1$ or
of $E_{1}^2\cup E_{2}^2\cup E_{3}^2$  are in the same column of $A$.

Consider the set $F_i=L_i\setminus ({E'}_{i}^1\cup {E'}_{i}^2)$, $i=1,2,3$. Note that $|F_i|=(k-3)q$, which is even and the members of $F_i$ can be paired as $x,-x$ for some $x\in F_i$.

Now we pair the $q$ $(k-3)$-subsets of $F_1$ with members of $S_2\setminus\{ E_{2}^1, E_{2}^2\}$ to obtain $k$-subsets, say $D_i^1$, $1\leq i\leq q$ as described in Case 1.
We also pair the $q$ $(k-3)$-subsets of $F_2$ with members of $S_1\setminus\{E_{1}^1, E_{1}^2\}$ to obtain $k$-subsets, say $D_i^2$, $1\leq i\leq q$.

Finally, we pair the $q$ $(k-3)$-subsets of $L_3$ with members of $S_3\setminus \{E_{3}^1, E_{3}^2\}$, to obtain $k$-subsets, say $D_i^3$, $1\leq i\leq q$. Note that the member sum of each of $k$-subset is zero.

By construction,
$${\cal P}_2=\{D_1^i,D_2^i,\ldots,D_q^i, D_{\ell-1}, D_{\ell}\mid 1\leq i\leq 3\}$$
is a partition of $X$  and is near orthogonal to the partition
${\cal P}_1=\{C_1,C_2,\ldots,C_n\}$.
So by Theorem \ref{mainconstruction}, there exists an $SMR(3n/k,n;k,3)$.
\end{proof}

\begin{example}\label {example 19}
In order to illustrate the construction given in Proposition \ref{n even and k odd}, Case 2 let consider the case $n=30$ and $k=9$. Then
$$\begin{array}{ll}
S_1=\{\{1,43,-44\},\{-1,-43,44\},\{2,38,-40\},\{-2,-38,40\}\};\\
S_2=\{\{4,22,-26\}, \{-4,-22,26\},\{5,20,-25\}, \{-5,-20,25\}\};\\
S_3=\{\{3,15,-18\},\{-3,-15,18\},\{9,21,-30\},\{-9,-21,30\}\}.\\
\end{array}$$

So the unused members of $R_1$, $R_2$ and $R_3$, respectively, are
$$\begin{array}{ll}
L_1=\{\pm7,\pm8,\pm10,\pm11,\pm13,\pm14,\pm16,\pm17,\pm19\};\\
L_2=\{\pm23,\pm28,\pm29,\pm31,\pm32,\pm34,\pm35,\pm37,\pm41\};\\
L_3=\{\pm6,\pm12,\pm24,\pm27,\pm33,\pm36,\pm39,\pm42,\pm45\}.
\end{array}$$

\noindent Note that since $k=9$, then $k/3-3=0$. So $E_1=E_2=E_3=\emptyset.$ Define
$$\begin{array}{ll}
D_1^1=\{7,-7,8,-8,10,-10\}\cup \{-5,-20,25\}      \\
D_1^2=\{11,-11,13,-13,14,-14 \}\cup \{-4,-22,26\}    \\
D_1^3=\{ 16,-16,17,-17,19,-19\}\cup \{5,20,-25\} \\
\end{array}$$

$$\begin{array}{ll}
D_2^1=\{23,-23,28,-28,29,-29\}\cup \{-1,-43,44\}\  \\
D_2^2=\{ 31,-31,32,-32,34,-34\}\cup \{-2,-38,40\} \\
D_2^3=\{35,-35,37,-37,41,-41\}\cup \{2,38,-40\} \\
D_1^3=\{6,-6,12,-12,24,-24\}\cup \{3,15,-18\}    \\
D_2^3=\{27,-27,33,-33,36,-36\}\cup\{9,21,-30\} \\
D_3^3=\{39,-39,42,-42,45,-45\}\cup\{-9,-21,30\}  \\
D_{10}=\{1,43,-44,4,22,-26,-3,-15,18\}.\\
\end{array}$$

It is straightforward to see that
$${\cal P}_2=\{D_1^i,D_2^i,D_3^i, D_{10}\mid 1\leq i\leq 3\}$$
is a partition of $X=\{\pm1.\pm2,\ldots,\pm45\}$  and is near orthogonal to the partition
${\cal P}_1=\{C_1,C_2,\ldots,C_{30}\}$.
So by Theorem \ref{mainconstruction}, there exists an $SMR(10,30;9,3)$.
\end{example}

\begin{figure}[ht]
$$\begin{array}{|c|c|c|c|c|c|c|c|c|c|c|c|}\hline
1&16&-17&-12&12&&&-6&6&-3&3&\\\hline
17&-1&&&-16&13&5&-5&-13&&8&-8\\\hline
&&2&-2&4&-9&9&&7&10&-11&-10\\\hline
-18&-15&15&14&&-4&-14&11&&-7&&18\\\hline
\end{array}$$
\caption{An $SMR(4,12;9,3)$}
		\label{4,12;9,3}
\end{figure}


By Propositions \ref{nodd3divk}, \ref{nodd3divm}, \ref{n and k even} and \ref{n even and k odd} we achieve the main theorem of this paper.

\begin{maintheorem}
Let $m,n,k$ be positive integers and $3\leq m,k\leq n$. Then there exists an $SMR(m,n;k,3)$
if and only if $mk=3n$.
\end{maintheorem}

\newpage

\begin{figure}[ht]
{\tiny
$$\begin{array}{|c|l|}\hline
(n,k) & \mbox {{\rm 3}-subsets of } R_1\cup R_2 \\ \hline
(10,5)  &M_1=\{1, 13, -14\},    M_1'=\{-1, -13, 14\},\\
       &  N_1=\{4, 7, -11\},     N_1'=\{-4, -7, 11\} \\ \hline
(12, 9) & M_1=\{1,16,-17\},     M_1'=\{-1,-16,17\} \\
        & N_1=\{4, 7, -11\},    N_1'=\{-4, -7, 11\} \\ \hline
(14,7) &  M_1=\{1, 19, -20\},   M_1'=\{-1, -19, 20\} \\
       &  N_1=\{4, 10, -14\},   N_1'=\{-4, -10, 14\} \\ \hline
       (18,9) &  M_1=\{1, 25, -26\},   M_1'=\{-1, -25, 26\} \\
       &  N_1=\{4, 13, -17\},   N_1'=\{-4, -13, 17\} \\ \hline
(20,5) &  M_1=\{1, 28, -29\},   M_1'=\{-1, -28, 29\} \\
       &  M_2=\{2, 23, -25\},   M_2'=\{-2, -23, 25\} \\
       &  N_1=\{4, 13, -17\},   N_1'=\{-4, -13, 17\} \\
       &  N_2=\{5, 11, -16\},   N_2'=\{-5, -11,16\} \\ \hline
(20,15) &  M_1=\{1, 28,-29\},   M_1'=\{-1,-28,29 \} \\
       &  N_1=\{4,13,-17\},   N_1'=\{-4,-13,17\} \\ \hline
(22,11) &  M_1=\{1, 31, -32\},   M_1'=\{-1, -31, 32\} \\
        &  N_1=\{4, 16, -20\},   N_1'=\{-4, -16, 20\} \\ \hline
(24,9) &  M_1=\{1, 34, -35\},   M_1'=\{-1, -34, 35\} \\
       &  M_2=\{2, 29, -31\},   M_2'=\{-2, -29, 31\} \\
       &  N_1=\{4, 16, -20\},   N_1'=\{-4, -16, 20\} \\
       &  N_2=\{5, 14, -19\},   N_2'=\{-5, -14,19\} \\ \hline
(26,13)&  M_1=\{1, 37, -38\},   M_1'=\{-1, -37, 38\} \\
       &  N_1=\{4, 19, -23\},   N_1'=\{-4, -19, 23\} \\ \hline
(28,7) &  M_1=\{1, 40, -41\},   M_1'=\{-1, -40, 41\} \\
       &  M_2=\{2, 35, -37\},   M_2'=\{-2, -35, 37\} \\
       &  N_1=\{4, 19, -23\},   N_1'=\{-4, -19, 23\} \\
       &  N_2=\{5, 17, -22\},   N_2'=\{-5, -17, 22\} \\ \hline
(28,21)&  M_1=\{1, 40, -41\},   M_1'=\{-1,-40,41\} \\
       &  N_1=\{4,19,-23\},     N_1'=\{-4,-19,23\} \\ \hline
(30,5) &  M_1=\{1, 43, -44\},   M_1'=\{-1, -43,  44\}\\
       &  M_2=\{2, 38, -40\},   M_2'=\{-2, -38, 40\}\\
       &  M_3=\{4, 31, -35\},   M_3'=\{-4, -31,  35\}\\
       &  N_1=\{7, 22, -29\},   N_1'=\{-7,-22, 29\}\\
       &  N_2=\{10, 16, -26\},  N_2'=\{-10, -16, 26\}\\
       & N_3=\{8, 20, -28\},    N_3'=\{-8, -20, 28\}\\ \hline
(30,9) &  M_1=\{1,43,-44\},     M_1'=\{-1,-43,44\}\\
       &  M_2=\{2,38,-40\},     M_2'=\{-2,-38,40\}\\
       &  N_1=\{4,22,-26\},     N_1'=\{-4,-22,26\}\\
       &  N_2=\{5,20,-25\},     N_2'=\{-5,-20,25\}\\ \hline
(30,15)&  M_1=\{1, 43, -44\},   M_1'=\{-1, -43,  44\}\\
      &  N_1=\{4, 22, -26\},    N_1'=\{-4,-22, 26\}\\ \hline
\end{array}$$
\caption{Small cases for Lemma \ref {Lem:S_1.S_2}}
\label{small cases}
}
\end{figure}

\noindent {\bf Acknowledgement:} The authors would like to thank the referees for their useful
suggestions and comments.

\newpage



\end{document}